\documentclass[10pt,electronic]{amsart}        
\usepackage[applemac]{inputenc}
\usepackage[T1]{fontenc}
\usepackage[small,it]{caption}
\usepackage[english]{babel}
\usepackage[small,nohug]{diagrams}
\usepackage{url}
\usepackage[hmargin=3.7cm,vmargin=3.5cm]{geometry} 
\usepackage{graphicx}                         
\usepackage{amsmath}
\usepackage{amsthm}                      
\usepackage{amssymb}
\usepackage{amscd}     
\usepackage{mathrsfs}
\usepackage{stmaryrd}
\usepackage{enumerate}

\usepackage[isbn=false,doi=false,eprint=false,url=false,style=authoryear,dashed=false]{biblatex}
\makeatletter
\newrobustcmd*{\parentexttrack}[1]{%
  \begingroup
  \blx@blxinit
  \blx@setsfcodes
  \blx@bibopenparen#1\blx@bibcloseparen
  \endgroup}
\AtEveryCite{%
  }
\makeatother
\renewcommand{\cite}{\parencite}
\DeclareNameAlias{sortname}{last-first}
\renewbibmacro*{volume+number+eid}{%
  \printfield{volume}%
  \setunit*{\addnbspace}
  \printfield{number}%
  \setunit{\addcomma\space}%
  \printfield{eid}}
 \DeclareFieldFormat[article]{volume}{\textbf{#1}}
\DeclareFieldFormat[article]{number}{{\mkbibparens{#1}},}
\DeclareFieldFormat{pages}{#1 pages}
\DeclareFieldFormat
  [article,inbook,incollection,inproceedings,patent,thesis,unpublished]
  {pages}{#1}
\renewbibmacro{in:}{}
\theoremstyle{plain}                                                           
\newtheorem{thm}{Theorem}[section]
\newtheorem{lem}[thm]{Lemma}
\newtheorem{prop}[thm]{Proposition}

\theoremstyle{definition}

\newtheorem{rem}[thm]{Remark}

\DeclareMathOperator{\SL}{SL}
\DeclareMathOperator{\GL}{GL}
\DeclareMathOperator{\MHS}{MHS}
\DeclareMathOperator{\Sym}{Sym}                  
\DeclareMathOperator{\Sp}{Sp}
\newcommand{\gr}{\mathfrak{gr}}
\newcommand{\ct}{\mathsf{ct}}
\newcommand{\ud}{\,\mathrm d}
\newcommand{\rt}{\mathsf{rt}}
\newcommand{\field}[1]{\ensuremath{\mathbf{#1}}}
\newcommand{\Q}{\ensuremath{\field{Q}}}        
\newcommand{\C}{\ensuremath{\field{C}}}
\newcommand{\sym}{\ensuremath{\mathbb{S}}}
\newcommand{\Z}{\ensuremath{\field{Z}}}

\newcommand{\A}{\mathcal{A}}
\newcommand{\M}{\mathcal{M}}
\newcommand{\MM}{\overline{\mathcal{M}}}
\renewcommand{\AA}{\mathbf{A}}

\newcommand{\G}{\mathbb G}

\newcommand{\even}{\mathrm{even}}

\newcommand{\PH}{{P\! H}}

\newcommand{\V}{\mathbb V}
\newcommand{\R}{\mathrm{R}}
\newcommand{\RH}{{R\hspace{-0.5pt} H}}
\newcommand{\pH}{{}^\mathfrak p \mathcal H}
\newcommand{\sm}{\mathrm{sm}}

\newcommand{\MHM}{\mathsf{MHM}}

\title{Tautological rings of spaces of pointed genus two curves of compact type}
\keywords{tautological ring, Faber conjectures, moduli of curves, Gromov--Witten theory, cohomology of moduli spaces}
\subjclass[2010]{14H10, 14C17, 32S60, 14N35, 14D07, 55R55}
\author{Dan Petersen}
\thanks{}
\address{}

\begin{document} 
 \maketitle   
 
 \begin{abstract}
 We prove that the tautological ring of $\M_{2,n}^\ct$, the moduli space of $n$-pointed genus two curves of compact type, does not have Poincar\'e duality for any $n \geq 8$. This result is obtained via a more general study of the cohomology groups of $\M_{2,n}^\ct$. We explain how the cohomology can be decomposed into pieces corresponding to different local systems and how the tautological cohomology can be identified within this decomposition. Our results allow the computation of $H^k(\M_{2,n}^\ct)$ for any $k$ and $n$ considered both as $\sym_n$-representation and as mixed Hodge structure/$\ell$-adic Galois representation considered up to semi-simplification. A consequence of our results is also that all even cohomology of $\MM_{2,n}$ is tautological for $n < 20$, and that the tautological ring of $\MM_{2,n}$ fails to have Poincar\'e duality for all $n \geq 20$. This improves and simplifies results of the author and Orsola Tommasi.
 \end{abstract}

\section{Introduction}

The \emph{tautological classes} on the moduli spaces $\MM_{g,n}$ of $n$-pointed stable genus $g$ curves can be defined, following \cite{relativemaps}, to be the algebraic cycle classes of `geometric origin'. Specifically, we declare the fundamental classes of all spaces $\MM_{g,n}$ to be tautological, that the pushforward of tautological classes along the natural maps
$$ \MM_{g,n+1}\times \MM_{g',n'+1} \to \MM_{g+g',n+n'}, \qquad \MM_{g-1,n+2} \to \MM_{g,n}, \qquad \MM_{g,n+1} \to \MM_{g,n}$$
should be tautological, and that intersections of tautological classes are tautological. In this way we obtain a collection of subrings $R^\bullet(\MM_{g,n}) \subset A^\bullet(\MM_{g,n})$, the tautological rings. 

We can also consider the images of the tautological rings in cohomology, denoted $\RH^\bullet(\MM_{g,n})$. Finally, if $U \subset \MM_{g,n}$ is Zariski open, then we let $R^\bullet(U) = \mathrm{Im}\left(R^\bullet(\MM_{g,n}) \to A^\bullet(U)\right)$, and we define similarly $\RH^\bullet(U)$.

Particularly interesting (for us) are the Zariski open sets $\M_{g,n}^\rt \subset \M_{g,n}^\ct \subset \MM_{g,n}$ parametrizing curves with \emph{rational tails} and of \emph{compact type}, respectively. A stable $n$-pointed genus $g$ curves is of compact type if its dual graph is a tree; equivalently, if its jacobian is an abelian variety. It has \emph{rational tails} if one of its components has geometric genus $g$, implying that all other components are trees (or `tails') of rational curves attached to this genus $g$ component.

Faber and Pandharipande \cite{faberconjectures,faberjapan,pandharipandequestions} proposed a uniform conjectural description of the tautological rings $R^\bullet(\M_{g,n}^\rt)$, $R^\bullet(\M_{g,n}^\ct)$ and $R^\bullet(\MM_{g,n})$. These conjectures naturally split into several smaller pieces. First were the \emph{vanishing} and \emph{socle} conjectures. These assert that 
\begin{align*}
&R^{g-2+n - \delta_{0,g}}(\M_{g,n}^\rt) \cong \Q,  & R^k(\M_{g,n}^\rt) =0 \text{ for } k > g-2+n - \delta_{0,g} , \\
&R^{2g-3+n}(\M_{g,n}^\ct) \cong \Q, & R^k(\M_{g,n}^\ct) =0 \text{ for } k > 2g-3+n, \\
&R^{3g-3+n}(\MM_{g,n}) \cong \Q&
\end{align*}
(and obviously, $R^k(\MM_{g,n}) =0 \text{ for } k > 3g-3+n$). Here $\delta_{0,g}$ is the Kronecker delta. Given these statements, one can (after choosing a generator for the top degree) describe the pairing into the top degree in terms of proportionalities. On $\MM_{g,n}$ these top degree intersection numbers are all determined by Witten's conjecture (Kontsevich's theorem). The second part of the conjectures were explicit expressions determining the top intersections also in the rational tails and compact type cases: the $\lambda_g\lambda_{g-1}$-conjecture and the $\lambda_g$-conjecture, respectively. 

The vanishing, socle and intersection number conjectures are now all theorems. This represents work of a large number of people, and all of the statements now have several different proofs, enlightening in their own way. See the survey \cite{pcmifaber} and the detailed references therein.

However, the final part of the conjectures is now known to be false in general. The \emph{perfect pairing} conjecture proposed that the pairing into the top degree is always perfect, so that the tautological rings enjoy Poincar\'e duality. A different way of stating this is in terms of the relations between the generators of the tautological rings: every potential relation between tautological classes that is consistent with the top degree pairing is actually a true relation. In \cite{petersentommasi} we showed that this conjecture fails on $\MM_{2,n}$. Before this, computer calculations had been used to find examples where it seems likely that the perfect pairing conjecture fails, also in the rational tails and compact type cases, see \cite{pcmifaber,yin,pixtonthesis}.

The main result of this paper is that the conjecture fails also on $\M_{2,n}^\ct$. In fact we observe  failure much sooner --- the tautological cohomology ring of $\MM_{2,n}$ fails to have Poincar\'e duality for the first time when $n=20$ (see Subsection \ref{stablecounterexample}), but in the compact type case the pairing fails to be perfect already when $n=8$. Even though computer assisted computations of tautological rings for small $g$ and $n$ have not yet gotten as far as $\M_{2,8}^\ct$, it is not {inconceivable} that one could actually determine the intersection matrices in this case with enough computing power and a clever implementation. Thus one could for instance hope to explicitly write down a nonzero tautological class that pairs trivially with all tautological classes in opposite degree. (The proof given here seems not explicit enough to produce such a class). Doing something similar for $\MM_{2,20}$ is utterly doomed to fail.

The basic strategy in \cite{petersentommasi} was to study instead the tautological cohomology ring $\RH^\bullet(\MM_{2,n})$, and to prove that its Betti numbers are not symmetric about the middle degree. This certainly implies that the tautological cohomology ring can not have a perfect pairing, and then also the tautological ring itself. This is what we do in this paper, too. 

The  results will follow from a more general study of the cohomology of $\M_{2,n}^\ct$. We approach the cohomology of $\M_{2,n}^\ct$ via the Leray spectral sequence for $f \colon \M_{2,n}^\ct \to \M_2^\ct$. Our first result, Theorem \ref{symmetrythm}, is that the Leray spectral sequence degenerates, and therefore that the cohomology of $\M_{2,n}^\ct$ can be expressed in terms of the cohomology of local systems on $\M_2^\ct$ and local systems supported on the boundary. Moreover, after the results in \cite{localsystemsA2} we actually know the cohomology of all these local systems in any degree, together with their mixed Hodge structure up to semi-simplification. This allows us to obtain very detailed information about the cohomology of $\M_{2,n}^\ct$: we can compute $H^k(\M_{2,n}^\ct)$ for any $k$ and $n$ considered both as $\sym_n$-representation and as mixed Hodge structure/$\ell$-adic Galois representation (considered up to semi-simplification). This has been implemented on a computer, in a way described in  Section \ref{computer}.

The results described in the preceding paragraph provide a decomposition of the cohomology of $\M_{2,n}^\ct$ into pieces corresponding to different local systems. Theorem \ref{mainthm} gives a description of  the subspace 
$$ \RH^\bullet(\M_{2,n}^\ct) \subset H^\bullet(\M_{2,n}^\ct)$$ 
in terms of this decomposition. Moreover, Poincar\'e--Verdier duality applied to the morphism $f \colon \M_{2,n}^\ct \to \M_2^\ct$ implies (Theorem \ref{counterexample}) that $\RH^\bullet(\M_{2,n}^\ct)$ can be written as a part that is symmetric about the middle degree (corresponding to the trivial local systems on $\M_2^\ct$ and on the boundary) and a part that is symmetric about a \emph{different} degree. The latter part is nonzero for any $n \geq 8$. 

The non-symmetric part of $\RH^\bullet(\M_{2,n}^\ct)$ is given by the image of a certain Gysin map. This Gysin map was studied previously in \cite{petersentommasi}, where we conjectured that it is never zero. In Section \ref{gysinsection} we present a proof of this conjecture. This non-vanishing result implies also that all even cohomology of $\MM_{2,n}$ is tautological for $n<20$, and that the tautological ring of $\MM_{2,n}$ fails to have perfect pairing for all $n \geq 20$ (Theorem \ref{stablethm}). Theorem \ref{stablethm} improves upon the results of \cite{petersentommasi}, with a simpler proof.

Our results imply in fact that the Betti numbers of $\RH^\bullet(\M_{2,n}^\ct)$ are always bigger above the middle degree than below it for $n \geq 8$, and similarly that the Betti numbers of $\RH^\bullet(\MM_{2,20})$ are always bigger above the middle for $n \geq 20$. 
It would be interesting if this phenomenon, that the tautological ring is bigger above the middle degree than below it, always were true. We remark that it is always true in cohomology on the space $\MM_{g,n}$. Indeed, according to the hard Lefschetz theorem, there is an isomorphism 
$$ H^{3g-3+n-q}(\MM_{g,n}) \to H^{3g-3+n+q}(\MM_{g,n})$$
defined by iterated multiplication by an ample class $\omega$. Now note that $\omega$ must be a tautological class (since it is a divisor), so multiplication by $\omega^q$ maps tautological classes to tautological classes. Thus the tautological cohomology below the middle degree injects into that above it. Perhaps a similar hard Lefschetz principle for tautological classes exists also on the spaces $\M_{g,n}^\rt$ and $\M_{g,n}^\ct$, and/or in the Chow ring. 

\subsection{Conventions} We work throughout over the complex numbers and use the language of mixed Hodge theory, but our arguments would work as well in \'etale cohomology. Cohomology groups and Chow groups are taken with $\Q$-coefficients unless stated otherwise.

\subsection{Acknowledgements}

This work was carried out in the group of Pandharipande at ETH Z\"urich, supported by grant ERC-2012-AdG-320368-MCSK. I am grateful to Rahul Pandharipande for his suggestions and his interest in this work. Conversations with Francesco Lemma and correspondence with J\"org Wildeshaus were very useful for Section \ref{gysinsection}. In making the computer implementation described in Section \ref{computer}, I benefited greatly from code written by Carel Faber. I also wish to thank an anonymous referee for noticing an error in the original version of Theorem \ref{symmetrythm}.

\section{Local systems and the Leray spectral sequence}

\subsection{Symplectic local systems}

To any $g$-tuple of integers $\lambda = (\lambda_1 \geq \lambda_2 \geq \ldots \geq \lambda_g \geq 0)$ we associate a rational representation of $\Sp(2g)$ whose highest weight is given by $\lambda$. This representation defines also a rational local system $\V_\lambda$ on the moduli stack $\A_g$ of principally polarized abelian varieties of dimension $g$. It naturally underlies a polarized variation of Hodge structure  of weight $\vert \lambda \vert = \lambda_1 + \ldots + \lambda_g$. We may pull back $\V_\lambda$ along the Torelli map $\M_g^\ct \to \A_g$; we denote its pullback by the same name. We call all such $\V_\lambda$ \emph{symplectic local systems}. 

The reason for the appearance of these local systems is that if $\pi \colon \mathcal C \to \M_g$ is the universal curve, then there is an isomorphism
$$ \R^1 \pi_\ast \Q \cong \V_{1,0,\ldots,0}.$$
The remaining symplectic local systems appear naturally within the tensor powers of this local system: if $(\V_{1,0,\ldots,0})^{\otimes n}$ is decomposed into simple local systems, then all $\V_\lambda$ with $\vert \lambda \vert \leq n$ and $\vert \lambda \vert \equiv n \pmod 2$ will appear in the decomposition (the ones with $\vert \lambda \vert < n$ appearing with a Tate twist in order to have weight $n$). Thus these are precisely the local systems appearing in the cohomology of the universal curve and its fibered powers.

In this paper we will only consider the cases $g \in \{1,2\}$. When $g=2$ we denote these local systems $\V_{l,m}$ (with $l \geq m \geq 0$). Note that in this particular case, the Torelli map $\M_2^\ct \to \A_2$ is in fact an isomorphism of stacks. When $g=1$ these local systems are considered on $\M_{1,1} \cong \A_1$ and denoted $\V_k$, $k \geq 0$. We will also have reason to consider local systems on $\Sym^2 \M_{1,1}$ given by representations of the semidirect product $(\SL(2) \times \SL(2) ) \rtimes \sym_2$ (see \cite{d-elliptic}), and we use the phrase `symplectic local system' for these as well.

The cohomology groups $H^\bullet(\M_{1,1},\V_k)$ are completely determined as mixed Hodge structures by the Eichler--Shimura isomorphism, or rather its Hodge-theoretic interpretation \cite{zuckerdegeneratingcoefficients}. The cohomology groups $H^\bullet(\M_2^\ct,\V_{l,m})$, considered as mixed Hodge structures up to semi-simplification, are calculated for all $l$ and $m$ in \cite{localsystemsA2}. 

\subsection{An application of the decomposition theorem} 

Our approach to the cohomology of the spaces $\M_{2,n}^\ct$ will be to consider the Leray spectral sequence for $f \colon \M_{2,n}^\ct \to \M_2^\ct$. An important observation is that $f$ is proper, so that we may apply the decomposition theorem of \cite{bbd}. The modest goal of this subsection is to spell out explicitly what the decomposition theorem says in this particular case; the result can be stated without mentioning perverse sheaves or mixed Hodge modules. The reader who is not familiar with such things can feel free to take Theorem \ref{symmetrythm} on faith and proceed to read the rest of the paper (or refer to \cite{decataldomigliorinisurvey,saitointroduction} as needed).



Let $i \colon \Sym^2 \M_{1,1} \to \M_2^\ct$ be the inclusion.

\begin{thm}\label{symmetrythm}
Let $f \colon \M_{2,n}^\ct \to \M_2^\ct$ be the forgetful map. 
\begin{enumerate}[\normalfont(i)]
\item There is an isomorphism $\R f_\ast \Q \cong \bigoplus_q \R^q f_\ast \Q[-q]$. In particular, the Leray spectral sequence for $f$ degenerates. 
\item We can write $\R^q f_\ast \Q \cong  A^q \oplus i_\ast B^q$, where $A^q$ and $B^q$ are polarized variations of Hodge structure of weight $q$ on $\M_2^\ct$ and on $\Sym^2 \M_{1,1}$, respectively. These are direct sums of Tate twists of symplectic local systems.
\item There are isomorphisms $A^{n+q} \cong A^{n-q}(-q)$ and $B^{n+1+q} \cong B^{n+1-q}(-q)$, where $(-q)$ denotes a Tate twist.
\end{enumerate}
\end{thm}

\begin{proof}
Note that $f$ is proper and $\M_{2,n}^\ct$ is smooth, so that there exists an isomorphism 
$$ \R f_\ast \Q\cong \bigoplus_q \pH^q(\R f_\ast \Q)[-q].$$
in $D^b(\MHM(\M_2^\ct))$, according to Saito's version of the decomposition theorem. Here we denote by $\pH^q(K)$ the usual (middle perversity) cohomology sheaves of a mixed Hodge module $K$. Moreover, each $\pH^q(\R f_\ast \Q)$ is a pure Hodge module, as $\R f_\ast$ preserves purity. 

According to the semisimplicity theorem, each $\pH^q(\R f_\ast \Q)$ is a direct sum of simple objects. Let ${\M_2^\ct} = \M_2 \cup \Sym^2(\M_{1,1})$ be the stratification of $\M_2^\ct$ according to topological type. Then $f$ is a stratified map (even a stratified submersion); the restriction of $f$ to the inverse image of a stratum is a fiber bundle. 
Since in addition both strata are smooth (in the sense of stacks), it is known that each of the simple summands of $\pH^q(\R f_\ast \Q)$ is obtained as an intermediate extension of a local system on one of the two strata. 

Let $\V$ be a local system on one of these strata, which occurs in the decomposition of $\R f_\ast \Q$, and let us consider its intermediate extension. If the stratum is $\Sym^2(\M_{1,1})$, which is closed in $\M_2^\ct$, then the intermediate extension is simply the usual pushforward. If the stratum is $\M_2$, then we claim that the local system extends to a local system on $\M_2^\ct$. Since $\M_2^\ct$ is smooth, this extension will then be equal to the intermediate extension. We should prove that the monodromy of $\V$ vanishes around $\Sym^2 \M_{1,1}$. But the monodromy is given by a Dehn twist around a vanishing cycle, and since we are in the case of compact type, all nodes of the curves involved will be separating, which means that we only consider Dehn twists around separating curves. But such a Dehn twist lies in the Torelli group, that is, it acts trivially on the cohomology of the universal curve (and then all its fibered powers, et cetera), from which the result follows. An equivalent way of thinking about this is that the local systems occurring on $\M_2$ will all be {symplectic local systems}, from which it is immediate that they extend to $\M_2^\ct$. In particular it follows that all perverse sheaves occuring in the decomposition of $\R f_\ast \Q$ are bona fide sheaves, not just complexes of sheaves. Thus there is also a decomposition
$$ \R f_\ast \Q\cong \R^q f_\ast \Q[-q]$$
for the classical (non-perverse) $t$-structure, the usual Leray spectral sequence degenerates, and so on. We have now proven (i) and (ii).  

For the last part, observe that $\Q[n+3]$ is a self-dual sheaf on $\M_{2,n}^\ct$, and that $\R f_\ast $ preserves self-duality. Thus  $$\pH^{n+3+q} (\R f_\ast \Q) \cong \pH^{n+3-q} (\R f_\ast \Q)^\vee$$
for all $q$ (Poincar\'e--Verdier duality). Recall that
$$ \pH^q(K) = \mathcal H^{q-d} (K)[d]$$ if the complex $K$ has cohomology sheaves which are local systems supported on a smooth closed subvariety of dimension $d$. By what we have written so far, we can write $\R f_\ast \Q = A \oplus i_\ast B$, where $A$ (resp.\ $B$) is a direct sum of symplectic local systems on $\M_2^\ct$ (resp.\ on $\Sym^2 \M_{1,1})$. Then $\mathcal H^{n+q}(A) \cong \mathcal H^{n-q}(A)^\vee$ and $\mathcal H^{n+1+q}(B) \cong \mathcal H^{n+1-q}(B)^\vee$. All representations of the symplectic group are self-dual, so a symplectic local system is always self-dual up to a Tate twist, which concludes the proof.
\end{proof}


\begin{rem} \label{decompositionsymmetry} In our proof, the symmetries $A^{n+q} \cong A^{n-q}(-q)$ and $B^{n+1+q} \cong B^{n+1-q}(-q)$ arose from the definition of the perverse $t$-structure. A more lowbrow way of seeing that these are the `right' symmetries is that they are necessary in order to have Poincar\'e duality on $\M_{2,n}^\ct$. Indeed, Poincar\'e duality says that
\begin{equation}\label{ett}
 H^{n+3+k}(\M_{2,n}^\ct) \cong H^{n+3-k}_c(\M_{2,n}^\ct)(-k), \end{equation}
and by  Theorem \ref{symmetrythm}(i),(ii) we have decompositions
\begin{equation}\begin{aligned}[m]
\label{tva} H^k(\M_{2,n}^\ct) &\cong \bigoplus_{p+q=k} H^p(\M_2^\ct,A^q) \oplus H^p(\Sym^2 \M_{1,1},B^q), \\
 H^k_c(\M_{2,n}^\ct) &\cong \bigoplus_{p+q=k} H^p_c(\M_2^\ct,A^q) \oplus H^p_c(\Sym^2 \M_{1,1},B^q).
\end{aligned}
\end{equation}The  symmetries of Theorem \ref{symmetrythm}(iii), together with Poincar\'e duality for local systems on $\M_2^\ct$ and on $\Sym^2 \M_{1,1}$, show that there are isomorphisms
$$ H^{3+p}(\M_2^\ct,A^{n+q}) \cong H^{3-p}_c(\M_2^\ct, A^{n-q})(-p-q) $$
as well as 
$$ H^{2+p}(\Sym^2 \M_{1,1},B^{n+1+q}) \cong H^{2-p}_c(\Sym^2 \M_{1,1},B^{n+1-q})(-p-q) $$
for all $p$ and $q$, which is exactly what is needed for compatibility of \eqref{ett} with \eqref{tva}. 
 \end{rem}

\subsection{Pure cohomology of the local systems}

Note that for the moment we are not primarily trying to understand all of the cohomology of $\M_{2,n}^\ct$, but rather the subspace of tautological cohomology. In particular, the  classes we are after lie in the even degree cohomology of $\M_{2,n}^\ct$ and hence we are only interested in the even cohomology of these local systems (since the hyperelliptic involution shows that $\R^q f_\ast \Q$ has vanishing cohomology for odd $q$). Moreover, algebraic cycle classes must be of pure weight (i.e.\ they must lie in the weight $p+q$ subspace of $H^p(\M_2^\ct,\R^q f_\ast \Q)$), and they must be of Tate type (i.e.\ with Hodge numbers along the diagonal).

The main theorem of \cite{localsystemsA2}, building on \cite{harder}, is the calculation of the cohomology of all symplectic local systems on $\M_2^\ct$, considered as mixed Hodge structures up to semi-simplification. In particular, we see the following:

\begin{prop}\label{puretate1}The only symplectic local systems on $\M_2^\ct$ with pure cohomology in even degree have the form $\V_{2a,2a}$ with $a \geq 0$. More specifically, there is the trivial local system $\V_{0,0}$, which has 
$$H^0(\A_2,\V_{0,0}) = \Q \qquad \text{and} \qquad H^2(\A_2,\V_{0,0}) = \Q(-1),$$
and for $a > 0$, 
$ H^2(\A_2,\V_{2a,2a}) $ is pure Tate and of the same dimension as the space of cusp forms for $\SL(2,\Z)$ of weight $4a+4$.\footnote{There is also a possible contribution in $H^4$ which is conjectured to always vanish and which we'll ignore for simplicity. The reader may check that its potential existence will in any case not affect our main results. Specifically, there is a  pure Tate class in $H^4(\A_2,\V_{2a,2a})$ for each normalized cusp eigenform for $\SL(2,\Z)$ of weight $4a+4$ such that the central value of the attached $L$-function vanishes. Conjecturally this should not happen, see Remark \ref{centralvalue}.}
\end{prop}

The local systems on $\Sym^2 \M_{1,1}$ will play a slightly smaller role, and rather than describe explicitly their cohomology in any degree we will be content to make the following easy observation:


\begin{prop}\label{puretate2} If a symplectic local system $\V_\lambda$ on $\Sym^2 \M_{1,1}$ has non-zero pure cohomology in even degrees, then it is either trivial --- in which case $H^0(\Sym^2 \M_{1,1},\Q) \cong \Q$ --- or has weight at least $20$, with pure cohomology in $H^2(\Sym^2 \M_{1,1},\V_\lambda)$. 
%
\end{prop}

The reason for the above proposition is that the pure cohomology in $H^1(\M_{1,1},\V_k)$ is the cohomology attached to cusp forms for $\SL(2,\Z)$ of weight $k+2$: for every normalized cusp eigenform $f$ of weight $k+2$ there is a $2$-dimensional summand $M_f \subset H^1(\M_{1,1},\V_k)$ of Hodge type $(k+1,0)$ and $(0,k+1)$. Pure cohomology in $H^1(\M_{1,1},\V_k)$ appears thus for the first time for $k=10$, corresponding to the discriminant cusp form $\Delta$ of weight $12$. Using the K\"unneth theorem and the injection
$$ H^\bullet(\Sym^2 \M_{1,1},\V_\lambda) \hookrightarrow H^\bullet(\M_{1,1} \times \M_{1,1},\V_\lambda), $$ we then see that pure cohomology in $H^2(\Sym^2 \M_{1,1},\V_\lambda)$ will not appear for local systems of weight below $20=10+10$. We remark that the pure cohomology in $H^2(\M_{1,1} \times \M_{1,1},\V_\lambda)$ is a sum of tensor products $M_{f_1} \otimes M_{f_2}$, and since $M_f \otimes M_f \cong \Sym^2 M_f \oplus \Q(-k-1)$, there can  even be cohomology of Tate type.


%

\section{Symmetries of the Betti numbers of the tautological ring}

\subsection{Tautological cohomology of local systems}

As in Theorem \ref{symmetrythm}(ii), we write $\R^q f_\ast \Q \cong A^q \oplus i_\ast B^q$, so that
\begin{equation}\label{directsum}
H^k(\M_{2,n}^\ct) \cong \bigoplus_{p+q=k} H^p(\M_2^\ct,A^q) \oplus H^p(\Sym^2 \M_{1,1},B^q).
\end{equation} The main result in this subsection is Theorem \ref{mainthm}, which characterizes the subspace of tautological cohomology of the left hand side in terms of the decomposition of the right hand side into different local systems. This result is the analogue for $\M_{2,n}^\ct$ of the following simpler results on $\M_{1,n}^\rt$ and $\M_{2,n}^\ct$: 

\begin{prop} \label{rationaltails} Consider the Leray spectral sequences for $\M_{1,n}^\rt \to \M_{1,1}$ and $\M_{2,n}^\rt \to \M_2$, respectively.
\begin{enumerate}[\normalfont(i)]
\item The tautological cohomology of $\M_{1,n}^\rt$ is exactly the cohomology associated to the trivial local system on $\M_{1,1}$.
\item The tautological cohomology of $\M_{2,n}^\rt$ is exactly the cohomology associated to the trivial local system on $\M_2$. 
\end{enumerate}\end{prop}

\begin{proof}The proof of (ii) is given in \cite[Section 4]{petersentommasi}. The proof of (i) is more or less identical. \end{proof}

\begin{rem} Since $\M_{1,1}$ and $\M_2$ both have the rational cohomology of a point, this is the same as asserting that the tautological rings are in both cases isomorphic to the monodromy invariant cohomology classes in a fiber of the respective fibrations.\end{rem}

Let $\Delta$ denote the boundary $\M_{2,n}^\ct \setminus \M_{2,n}^\rt$, and let $\widetilde \Delta$ denote its normalization (which is a disjoint union of spaces of the form $\M_{1,k+1}^\rt \times \M_{1,n-k+1}^\rt$). Since $\widetilde \Delta$ is smooth, there is a pushforward map in cohomology, $H^{k-2}(\widetilde \Delta )(-1) \to H^k(\M_{2,n}^\ct)$.

\begin{lem} The sequence 
$$ H^{k-2}(\widetilde \Delta )(-1) \to H^k(\M_{2,n}^\ct) \to H^k(\M_{2,n}^\rt) $$
is exact in the middle. In the decomposition of Equation \eqref{directsum}, the image of the first map consists of all summands of the form $H^p(\Sym^2 \M_{1,1},B^q)$ and the images of all Gysin maps 
$$H^{p-2}(\Sym^2 \M_{1,1},A^q)(-1) \to H^p(\M_2^\ct,A^q),$$ 
where $p+q=k$.\end{lem}

\begin{proof}
It will be easier to prove the Poincar\'e dual assertions in compactly supported cohomology (see Remark \ref{decompositionsymmetry} for how the decomposition in Equation \eqref{directsum} behaves under Poincar\'e duality). Part of the excision long exact sequence reads
$$ \ldots \leftarrow H^k_c(\Delta) \leftarrow H^k_c(\M_{2,n}^\ct) \leftarrow H^k_c(\M_{2,n}^\rt) \leftarrow \ldots, $$
so if we can show that $H^k_c(\Delta) \to H^k_c(\widetilde \Delta)$ is injective, then the first part will be proven. By Theorem \ref{symmetrythm}(ii), $\R^q f_\ast \Q \vert_{\Sym^2 \M_{1,1}}$ is pure of weight $q$. This implies that there is an injection
$$ \R^q f_\ast \Q \vert_{\Sym^2 \M_{1,1}} \hookrightarrow \R^q g_\ast \Q,$$
since the pure cohomology of a proper variety injects into the cohomology of a resolution of singularities \cite[Proposition 8.2.5]{hodge3}. This injection is in fact a direct summand, since polarized variations of pure Hodge structure are semisimple. Thus  $H^k_c(\Delta) \to H^k_c(\widetilde \Delta)$ is injective, by the Leray spectral sequence. 

The second part of the lemma follows now easily from the fact that the maps
$$H^k_c(\Delta) \leftarrow H^k_c(\M_{2,n}^\ct) \leftarrow H^k_c(\M_{2,n}^\rt)$$ 
are compatible with the maps of Leray spectral sequences:
\begin{equation*}
 H^p_c(\Sym^2 \M_{1,1}, \R^q f_\ast \Q) \leftarrow H^p_c(\M_2^\ct,\R^q f_\ast \Q) \leftarrow H^p_c(\M_2,\R^q f_\ast \Q). \qedhere \end{equation*}  \end{proof}

\begin{thm}\label{mainthm} In the Leray spectral sequence for $f \colon \M_{2,n}^\ct \to \M_2^\ct$, the tautological cohomology coincides with the contributions from the trivial local system on $\M_2^\ct$, the trivial local system on $\Sym^2 \M_{1,1}$, and from the image of the Gysin map $H^0(\Sym^2 \M_{1,1},\V_{2a,2a})(-1) \to H^2(\M_2^\ct,\V_{2a,2a})$. \end{thm}

\begin{proof}The tautological cohomology is pure and in even degree. According to Theorem \ref{symmetrythm}(ii) and Propositions \ref{puretate1}, \ref{puretate2}, there are four sources of pure even cohomology in the Leray spectral sequence:
\begin{enumerate}
\item $H^0$ of the trivial local system on $\M_2^\ct$;
\item $H^2(\M_2^\ct,\V_{2a,2a})$ is pure Tate, and has a basis given by the set of normalized Hecke eigenforms of weight $4a+4$ for $\SL(2,\Z)$ when $a>0$, and is one-dimensional for $a=0$; 
\item $H^0$ of the trivial local system on $\Sym^2 \M_{1,1}$;
\item $H^2(\Sym^2 \M_{1,1},\V_\lambda)$ for $\vert \lambda \vert \geq 20$.
\end{enumerate}
We should prove that the tautological part is given exactly by (1), (3) and the part of (2) which is in the image of the Gysin map. 

The previous lemma shows that (4), (3) and the part of (2) which is in the image of the Gysin map is precisely the cohomology which is pushed forward from $\widetilde \Delta$. Now if we consider the Leray spectral sequence for $g \colon \widetilde \Delta \to \Sym^2 \M_{1,1}$, then the classes of the form (2) and (3) appear in the cohomology of  the trivial local system on $\Sym^2 \M_{1,1}$, whereas the ones of the form (4) come from nontrivial $\V_\lambda$. It therefore follows from Proposition \ref{rationaltails}(i) that the classes of the form (3) and the part of (2) which is in the image of the Gysin map consists of tautological classes, being pushforwards of tautological classes on $\widetilde \Delta$. On the other hand, the classes of the form (4) are not tautological, since their restrictions to $\widetilde \Delta$ under $H^k(\M_{2,n}^\ct) \to H^k(\widetilde \Delta)$ are nontautological, again by  Proposition \ref{rationaltails}(i). We could also have used the main result of \cite{genusone} in this paragraph.

If we instead consider the restriction map $H^\bullet(\M_{2,n}^\ct) \to H^\bullet(\M_{2,n}^\rt)$, then we see that the classes of the form (2) which are \emph{not} in the image of the Gysin map are all non-tautological, as they restrict to nontautological classes in the interior. Indeed, these classes come from $H^2(\M_2,\V_{2a,2a})$ in the Leray spectral sequence for $\M_{2,n}^\rt \to \M_2$, so non-tautologicality follows from Proposition \ref{rationaltails}(ii). Finally, it is not hard to verify that the classes of the form (1) are all tautological, using that they restrict isomorphically onto the tautological ring of $\M_{2,n}^\rt$. This concludes the proof.\end{proof}

\begin{rem} \label{remarkbefore} The preceding theorem shows that we should try to understand the Gysin map $$H^0(\Sym^2 \M_{1,1},\V_{2a,2a})(-1) \to H^2(\M_2^\ct,\V_{2a,2a}).$$ Using the branching formula for $(\SL_2 \times \SL_2) \rtimes \sym_2 \subset \Sp(4)$ \cite{d-elliptic}, it is easy to see that $H^0(\Sym^2 \M_{1,1},\V_{2a,2a})(-1) \cong \Q(-2a-1)$ is one-dimensional. The main question is therefore whether or not this Gysin map is zero or not. In Theorem \ref{gysintheorem} we will prove that the Gysin map $H^0(\Sym^2 \M_{1,1},\V_{2a,2a})(-1) \to H^2(\M_2^\ct,\V_{2a,2a})$ is nonzero for all $a \neq 1$. Note that $H^2(\M_2^\ct,\V_{2a,2a})$ is nonzero in all cases except $a=1$, where this cohomology group vanishes. Thus the theorem asserts that the Gysin map is nonzero whenever it can be, and that its image is one-dimensional in these cases. \end{rem}



We postpone the proof of Theorem \ref{gysintheorem}, since the result and its proof is of a rather different flavor than the rest of the paper. Instead, let us prove the main result of this article.

\begin{thm} \label{counterexample}The tautological cohomology ring of $\M_{2,n}^\ct$ has symmetric Betti numbers if and only if $n \leq 7$. \end{thm}

\begin{proof}We analyze the contributions from the local systems in Theorem \ref{mainthm} separately:
\begin{enumerate}[(i)]
\item The trivial local system on $\M_2^\ct$ appears with the same multiplicity in $A^{n+q}$ and $A^{n-q}$, and has one-dimensional $H^0$ and $H^2$. Thus the cohomology coming from this local system is symmetric about degree $n+1$, as required.
\item The trivial local system on $\Sym^2 \M_{1,1}$ appears with the same multiplicity in $B^{n+1+q}$ and $B^{n+1-q}$, and has cohomology only in degree $0$. Hence the cohomology from this local system, too, is symmetric about degree $n+1$.
\item The local systems $\V_{2a,2a}$, $a > 0$, appear with the same multiplicity in $A^{n+q}$ and $A^{n-q}$. For $a > 1$ there is a $1$-dimensional subspace of $H^2(\M_2^\ct,\V_{2a,2a})$ which gives rise to tautological cohomology according to Remark \ref{remarkbefore} (and Theorem \ref{gysintheorem}). Hence this contribution is symmetric about degree $n+2$, and as soon as one of these local systems appear as a summand in $\R f_\ast \Q$, the Betti numbers of the tautological ring will be asymmetric. This happens for the first time when $n=8$ for $\V_{4,4} \subset \R^8 f_\ast \Q$, and then for all further values of $n$. \qedhere
\end{enumerate}

\end{proof}

\begin{rem} In an earlier preprint version of this paper, Theorem \ref{gysintheorem} was not available. Instead, I gave a direct  proof of the fact that the pure cohomology of $\M_{2,8}^\ct$ is tautological, using a rational parametrization of the moduli space. In light of the results of this section, this proves indirectly that $H^0(\Sym^2 \M_{1,1},\V_{4,4})(-1) \to H^2(\M_2^\ct,\V_{4,4})$ is nonzero and then that the tautological ring of $\M_{2,n}^\ct$ fails to have Poincar\'e duality for all $n \geq 8$. For the proof of Theorem \ref{stablethm} in the next subsection, the stronger statement of Theorem \ref{gysintheorem} is necessary.   \end{rem}

\subsection{Stable curves of genus two revisited} \label{stablecounterexample}

The results of the previous subsection can be used also to study the failure of the Gorenstein property for the tautological cohomology of $\MM_{2,n}$. The main result of \cite{petersentommasi} was that if $N$ is the smallest integer such that $\MM_{2,N}$ has non-tautological even cohomology, then the tautological cohomology ring of $\MM_{2,N}$ does not have Poincar\'e duality.  

We can now give a simpler proof of a stronger result: 

\begin{thm}\label{stablethm} If $n < 20$, then all even cohomology of $\MM_{2,n}$ is tautological (and thus the tautological cohomology satisfies Poincar\'e duality). For $n \geq 20$, the tautological cohomology does not have symmetric Betti numbers. \end{thm}

\begin{proof}Observe first of all that the even cohomology of $\MM_{2,n}$ has symmetric Betti numbers by Poincar\'e duality. Thus the theorem will be proven if we can show that the \emph{non-tautological} even cohomology of $\MM_{2,n}$ has non-symmetric Betti numbers for $n \geq 20$, and vanishes for $n< 20$.

By combining \cite[Corollaire 8.2.8]{hodge3} and \cite[Corollaire 3.2.17]{hodge2}, there is an exact sequence
$$ H^{k-2}(\MM_{1,n+2})(-1) \to H^k(\MM_{2,n}) \to W_k H^k (\M_{2,n}^\ct) \to 0, $$
for all $k$ and $n$. If $k$ is even, then $H^{k-2}(\MM_{1,n+2})(-1)$ consists only of tautological classes by \cite{genusone}, so the non-tautological cohomology in $H^k(\MM_{2,n})$ maps isomorphically onto the non-tautological cohomology in $W_k H^k (\M_{2,n}^\ct)$. We know the pure even non-tautological cohomology in $H^k (\M_{2,n}^\ct)$ according to Theorem \ref{mainthm}: in the Leray spectral sequence, it consists of those classes in $H^2(\M_2^\ct, \V_{2a,2a})$ which are not in the image of the Gysin map from $\Sym^2 \M_{1,1}$, and the pure subspace of $H^2(\Sym^2 \M_{1,1}, \V_{\lambda})$ when $\vert \lambda \vert \geq 20$. 

By the same argument as in the proof of Theorem \ref{counterexample}, the classes coming from $H^2(\Sym^2 \M_{1,1}, \V_{\lambda})$ when $\vert \lambda \vert \geq 20$ are symmetric about degree $n+3$ in $H^\bullet (\M_{2,n}^\ct)$, so that these classes actually have symmetric Betti numbers when considered in the cohomology ring of $\MM_{2,n}$. These classes do not appear when $n<20$.

However, the contribution from $H^2(\M_2^\ct, \V_{2a,2a})$ is symmetric about degree $n+2$, so as soon as this contribution is non-zero, the tautological ring will have asymmetric Betti numbers. This happens as soon as  the Gysin map
$$H^0(\Sym^2 \M_{1,1},\V_{2a,2a})(-1) \to H^2(\M_2^\ct, \V_{2a,2a})$$
fails to be surjective. According to Theorem \ref{gysintheorem}, this Gysin map always has one-dimensional image, which implies that it fails to be surjective for the first time when $a=5$ (for the local system $\V_{10,10}$), in which case $\dim H^2(\M_2^\ct,\V_{10,10}) = 2$, as there are two distinct cusp eigenforms for $\SL(2,\Z)$ of weight $24$. \end{proof}

\begin{rem} We also obtain a simple proof of the fact that all even cohomology on $\MM_{2,20}$ that is pushed forward from the boundary $\MM_{2,20} \setminus \M_{2,20}$ is tautological. This was proven in a somewhat cumbersome way in \cite[Section 5]{petersentommasi}. The non-tautological even cohomology from the boundary is exactly the cohomology coming from $H^2(\Sym^2 \M_{1,1}, \V_{\lambda})$ when $\vert \lambda \vert \geq 20$ in the Leray spectral sequence. Now note that when $n=20$, the only non-tautological even cohomology class on (the normalization of) the boundary lies in $H^{11}(\MM_{1,11}) \otimes H^{11}(\MM_{1,11})$, so its pushforward (if nonzero) would give a non-tautological class in degree $24$, and no other degree. But according to the proof of Theorem \ref{stablethm}, the cohomology coming from $H^2(\Sym^2 \M_{1,1}, \V_{\lambda})$ should be symmetric about degree $23$, a contradiction.\end{rem}

\begin{rem} According to \cite{graberpandharipande}, the algebraic cycle on $\MM_{2,20}$ which parametrizes bi-elliptic genus two curves for which the bi-elliptic involution switches the $20$ markings pairwise, is non-tautological in cohomology. According to the preceding theorem, their cycle is in fact the simplest example of a non-tautological algebraic cycle in genus two, as there are no non-tautological even cohomology classes for $n<20$, and when $n=20$,
the non-tautological even cohomology is concentrated in degree $22$. We note also that if we write the cohomology in degree $22$ (non-canonically) as the direct sum of tautological and non-tautological cohomology, then the non-tautological summand consists of a single copy of the representation $[2,2,\ldots,2]$ of $\sym_{20}$. Indeed, the local system $\V_{10,10}$ appears with this $\sym_{20}$-representation by Schur--Weyl duality. \end{rem}

\newcommand{\FM}{\mathsf{FM}}
\section{Computing the cohomology of $\M_{2,n}^\ct$ }\label{computer}

Let us now explain how the results of this paper can be used to compute the cohomology of $\M_{2,n}^\ct$ in arbitrary degree. Although we make heavy use of the theory of mixed Hodge modules \cite{saitomixedhodge}, we only use rather formal properties of Saito's theory. As in Theorem \ref{symmetrythm} we let $ \R^\bullet f_\ast \Q = A^\bullet \oplus B^\bullet, $
where $A^\bullet$ consists of local systems supported on $\M_2^\ct$ and $B^\bullet$ has support on $\Sym^2 \M_{1,1}$.

\subsection{The summand $A^\bullet$: review of results of Getzler}

\newcommand{\e}{\mathsf{e}}
\newcommand{\Var}{\mathrm{Var}}

Let $S$ be a scheme. The \emph{Grothendieck group of varieties over $S$}, $K_0(\Var_S)$, is the free abelian group generated by isomorphism classes of finite type schemes over $S$, modulo the following relation: whenever $Z \hookrightarrow X$ is a closed immersion, we have $$ [X] = [Z] + [X \setminus Z].$$
Taking $Z=X_{\mathrm{red}}$ shows that we may restrict our attention to reduced $X$. The Grothendieck group of varieties becomes a ring with the multiplication 
$$ [X] \cdot [Y] = [X \times_S Y].$$
If $S$ is a complex algebraic variety, let $\MHM(S)$ denote the abelian category of mixed Hodge modules on $S$. If $X$ is a finite type reduced scheme over $S$, then we put
$$ \e_S(X) = [Rf_!\Q] \in K_0(\MHM(S)),$$
where $f \colon X \to S$ is the structure morphism. This is the `Euler characteristic' or `Hodge--Deligne polynomial' of $X$ in this category. It satisfies
$$ \e_S(X) = \e_S(X \setminus Z) + \e_S(Z)$$
for $Z \subset X$ a closed subvariety, and $\e_S(X \times_S Y)=\e_S(X)\e_S(Y)$. We thus have a ring homomorphism $K_0(\mathrm{Var}_S) \to K_0(\MHM(S))$. There is a natural $\lambda$-ring structure on $K_0(\MHM(S))$: if $M$ is a mixed Hodge module, then $\lambda^n ([M]) = [\wedge^m M]$.

\begin{rem} \label{purityremark} If $S$ is smooth and $X \to S$ is smooth and proper, then Saito's theory implies that we can recover $\R^if_\ast \Q$ from $\e_S(X)$ by applying $\gr^W_i$.\end{rem}

We denote by $\Lambda$ the completed ring of symmetric functions, $$ \Lambda = \prod_{n\geq 0} \Lambda^n. $$
Each $\Lambda^n$ is isomorphic to the ring of virtual representations of $\sym_n$. As in \cite[Theorem 4.8]{mixedhodge} we can identify $K_0(\MHM(S)) \otimes \Lambda^n$ with the Grothendieck group of $\sym_n$-representations in $\MHM(S)$. If $X$ is a variety over $S$ with $\sym_n$-action, we denote by $\e^{\sym_n}_S(X)$ its class in $K_0(\MHM(S)) \otimes \Lambda^n$. Elements of $\Lambda$ are sequences of virtual representations and can therefore be thought of as `virtual $\sym$-modules', where by an $\sym$-module we mean a representation of the groupoid $\sym = \coprod_{n \geq 0} \sym_n$. We can identify $K_0(\MHM(S)) \widehat{\otimes} \Lambda$ (completed tensor product) with virtual $\sym$-modules in $\MHM(S)$. This, too, is naturally a $\lambda$-ring.

If $X$ is a variety over $S$, then we denote by $F(X/S,n)$ the \emph{relative configuration space} of $n$ distinct ordered points on $X$, that is, the complement of the `big diagonal' in the $n$-fold fibered product of $X$ with itself over $S$. Define an element $F_{X/S} \in K_0(\MHM(S)) \widehat\otimes \Lambda$ by
$$ F_{X/S} = \sum_{n \geq 0} \e^{\sym_n}_S(F(X/S,n)).$$
The element $F_{X/S}$ has the following explicit formula:
\begin{prop}[Getzler] \label{getzler1}
$ F_{X/S} = \exp \left(\sum_{n\geq 1} \sum_{d \mid n} \frac{\mu(n/d)}{n} \log(1+p_n) \psi_d \e_S(X) \right).$\end{prop}
Here $p_n \in \Lambda^n$ is the $n$th power sum, and $\psi_d$ denotes the $d$th Adams operation (that is, the $\lambda$-ring operation corresponding to $p_d$). This is proven when $S = \mathrm{Spec}(\C)$ in \cite[Section 5]{mixedhodge} (see also the treatment in \cite[Theorem 3.2]{getzlerpandharipandebetti}) and for $X \to S$ any morphism of quasi-projective complex varieties in \cite[Theorem 4.5]{getzler99}. The former proof uses only formal properties of the Euler characteristic $\e \colon K_0(\Var_\C) \to K_0(\MHS_\Q)$, whereas the latter proof is more involved and involves a `by hand' construction of a complex in $D^b(\MHM(X^n))$ which resolves $j_!j^\ast \Q$, where $j$ is the open embedding of $F(X/S,n)$ in $X^n$, and $X^n$ is the $n$-fold fibered power over $S$. We remark, however, that the former proof actually works equally well in the relative setting (and this removes the quasi-projectivity assumption), by replacing $\e$ by its relative version $\e_S \colon K_0(\Var_S) \to K_0(\MHM(S))$.

Let us now in addition assume that $S$ is smooth and irreducible, and that $X \to S$ is a smooth, quasi-projective morphism. Then we denote by $\FM(X/S,n)$ the \emph{relative Fulton--MacPherson compactification} of the configuration space of $n$ distinct ordered points on $X$, as defined in \cite{fmcompactification} when $S$ is a point and in \cite[Section 1]{pandharipanderelative} in the relative case. Under our hypotheses, the fiber of $\FM(X/S,n)$ over a point $s \in S$ is exactly the usual Fulton--MacPherson compactification of $n$ points in the fiber $X_s$. Define a second element
$$ \FM_{X/S} = \sum_{n \geq 0} \e^{\sym_n}_S(\FM(X/S,n)). $$
Finally let 
$$ \MM = \sum_{n \geq 3} \e^{\sym_n}(\MM_{0,n}) \in K_0(\MHS_\Q) \widehat\otimes \Lambda.$$
The element $\MM$ was expressed in terms of Tate type Hodge structures in \cite{getzler94}. Note that $ K_0(\MHM(S))$ is in a natural way a $K_0(\MHS_\Q)$-algebra.

If $f$ is a symmetric function then we denote by $f^\perp \colon \Lambda \to \Lambda$ the operation which is adjoint to multiplication by $f$, as well as its extension to $R \widehat{\otimes} \Lambda \to R \widehat{\otimes} \Lambda$, for any ring $R$. The plethysm operation on $\Lambda$ is denoted by $\circ$. If $R$ is in addition a $\lambda$-ring, then the plethysm extends in a natural way to an operation $\circ$ on $R \widehat{\otimes} \Lambda$.

The following result is proven in \cite[Proposition 6.9]{mixedhodge}.

\begin{prop}[Getzler] \label{getzler2} $\FM_{X/S} = F_{X/S} \circ (s_1 + s_1^\perp \MM)$.  \end{prop}

\begin{proof}The proof uses a certain yoga of symmetric functions and $\sym_n$-representations. In this formalism, the plethysm should be interpreted as `attaching', in a sense which can be made precise using Joyal's theory of species.

Seen in this framework, the equation simply states that to obtain a point of one of the Fulton--MacPherson compactifications, we should first choose a configuration of distinct markings in a fiber of $X \to S$, and at each of these markings we should either attach `a point' (the summand $s_1$) or `a stable pointed tree of projective lines with one distinguished point' (the summand $s_1^\perp \MM$), the distinguished point being the one which is attached to $X$. 

If $[V] \in \Lambda^n$ is the class of a representation of $\sym_n$, then $s_1^\perp([V]) = [\mathrm{Res}^{\sym_n}_{\sym_{n-1}} V] \in \Lambda^{n-1}$, which explains why the operator $s_1^\perp$ corresponds to choosing one of the markings as distinguished.\end{proof}

In our case, we take $S = \M_2$, $X = \M_{2,1}$, and $f$ the natural forgetful map. Then $F(X/S,n) = \M_{2,n}$ and $\FM(X/S,n) = \M_{2,n}^\rt$. We have 
$$ \e_S(X) = \Q - \V + \Q(-1)$$
with  $\V = \V_{1,0} \cong \R^1f_\ast \Q$. The only other mixed Hodge modules which will appear from now on will arise by applying Adams operations to this expression. Hence for our purposes we can replace $K_0(\MHM(S))$ with the representation ring of $\mathrm{GSp}(4)$ --- the sub-$\lambda$-ring of $K_0(\MHM(S))$ generated by $\V$ and $\Q(-1)$ is isomorphic to this representation ring, by an isomorphism sending $\V$ to the standard $4$-dimensional representation and $\Q(-1)$ to the multiplier. 

Propositions \ref{getzler1} and \ref{getzler2}, and the calculation of $\MM$ from \cite{getzler94}, allow us to calculate as many terms as we wish of $\FM_{X/S}$ in terms of the local systems $\V_{l,m}$ and their Tate twists. By Remark \ref{purityremark}, the same is true for each of the $\R^i f_\ast \Q$ (considered as $\sym_n$-equivariant Hodge modules on $\M_2$), where $f$ denotes the projection $\M_{2,n}^\rt \to \M_2$.

\subsection{The summand $B^\bullet$}

To ease notation in this section, we let $T = \M_2^\ct \setminus \M_2 \cong \Sym^2 \M_{1,1}$, and $Y = \M_{2,1}^\ct \setminus \M_{2,1}$, the universal curve over $T$. As in the previous subsection we have the relative configuration space $F(Y/T,n)$; it can be identified with a Zariski open subset of $\M_{2,n}^\ct \setminus \M_{2,n}^\rt$. Although we do not give a general definition of the Fulton--MacPherson compactification of a singular variety, we make the definition
$$ \FM(Y/T,n) = \M_{2,n}^\ct \setminus \M_{2,n}^\rt,$$
by analogy with the preceding subsection. 

To determine the summand $B^\bullet$ we will compute 
$$ \FM_{Y/T} = \sum_{n\geq 0} \e^{\sym_n}_T(\FM(Y/T,n)) \in K_0(\MHM(T))\widehat \otimes \Lambda. $$
Each $\e_T(\FM(Y/T,n))$ is equal to the sum of $B^\bullet$ and the restriction $A^\bullet\vert_{\Sym^2 \M_{1,1}}$. The latter is determined by the branching formula from $\Sp(4)$ to $(\SL(2) \times \SL(2)) \rtimes \sym_2$ \cite{d-elliptic} since we have already expressed $A^\bullet$ in terms of local systems attached to representations of $\Sp(4)$ and their Tate twists, and in this way we find $B^\bullet$.

We can write \cite[Lemma 5.2]{petersentommasi}
$$ \e_T(Y) = \Q - \V + \Q(-1) + \Q(-1) \otimes \varepsilon,$$
where the term $\V$ denotes the restriction of the local system $\V = \V_{1,0}$ we had on $\M_2^\ct$, and where $\varepsilon$ is defined by the equality $\mu_\ast \Q = \Q \oplus \varepsilon$, with $\mu \colon (\M_{1,1})^2 \to \Sym^2 \M_{1,1}$ the double cover. As a local system, $\varepsilon$ is defined by the sign representation of $\sym_2$.  The terms $\Q(-1) + \Q(-1) \otimes \varepsilon$ simply mean that $H^2(Y_t)$ is $2$-dimensional for any $t \in T$, and it is spanned by the sum of the fundamental classes of the two components (which is $\sym_2$-invariant) and the difference of the fundamental classes (which transforms under the sign representation). 

Let $Y^\sm\subset Y$ be the locus where the morphism $Y \to T$ is smooth; that is, the complement of the node in each fiber.

\begin{prop} $\FM_{Y/T} = \left( F_{Y^\sm/T} \circ (s_1 + s_1^\perp \MM) \right)\cdot (1+\varepsilon \otimes s_{1,1}^\perp \MM + s_2^\perp \MM) $.\end{prop}

\begin{proof}
Again we use the yoga of species. As in the preceding proposition, we can interpret $F_{Y^\sm/T} \circ (s_1 + s_1^\perp \MM)$ as a choice of distinct markings in the smooth locus of a fiber of $Y \to T$, and for each of them we attach a point or a pointed tree of projective lines with one of the points distinguished. 

In order to obtain $\FM_{Y/T}$ we should also put markings at the node. Here we can either have no marking at all (the summand `$1$'), or we need to choose a pointed tree of projective lines with \emph{two} of the points distinguished. Choosing two marked points is the same as restricting an $\sym_n$-representation $V$ to $\sym_{n-2}\times \sym_2$. The part of this restriction that transforms according to the trivial representation of $\sym_2$ is $s_2^\perp [V]$, and the part that transforms according to the sign representation is $s_{1,1}^\perp [V]$. This explains the formula. 
\end{proof}

Note that even though the map $h \colon \FM(Y/T,n) \to T$ fails to be smooth, each $\R^i h_\ast \Q$ is still pure of weight $i$. Indeed, this follows from Theorem \ref{symmetrythm}. One could also prove this by writing the fiber of $h$ over $t$ as an iterated blow-up of $(Y_t)^n$ (which clearly has pure cohomology) in loci which themselves have pure cohomology. Thus $\FM_{Y/T}$ determines each of the terms $\R^i h_\ast \Q$.

The expressions for $\FM_{X/S}$ that were derived earlier will only contain Schur functors applied to $\V$ and $\Q(-1)$. Here, we are going to find an expression for $\FM_{Y/T}$  in terms of Schur functors applied to $\V$, $\Q(-1)$ and $\varepsilon$. Moreover, both $\FM_{X/S}$ and $\FM_{Y/T}$ are going to be `effective', in the sense that terms of odd weight always occur with a negative coefficient and terms of even weight with a positive coefficient. Thus it makes sense to say that $\FM_{X/S}$ is a `direct summand' of $\FM_{Y/T}$. (It would perhaps be more accurate to talk about `$i^\ast \R^0 j_\ast \FM_{X/S}$' as such a summand, where $j \colon S \to \M_2^\ct$ and $i \colon T \to \M_2^\ct$ are the respective open and closed immersions.) In making our calculations we have verified in particular that $\FM_{Y/T} - \FM_{X/S}$ satisfies the Poincar\'e duality with a degree shift of Theorem \ref{symmetrythm}, a nontrivial consistency check. 

In Table \ref{tabell1} we tabulate the cohomology of $\M_{2,n}^\ct$ for some small values of $n$. We encode the cohomology as the polynomial $\sum_i [H^i(\M_{2,n}^\ct)] \cdot t^i $, where  $[H^i(\M_{2,n}^\ct)] $ is considered as an element of $K_0(\MHS_\Q) \otimes \Lambda^n$. This class is given as a polynomial in $L$ and the Schur polynomials $s_\lambda$. By $L$ we mean the Tate Hodge structure $\Q(-1)$. (For $n \leq 9$ all the cohomology is of Tate type.) In general we can only determine the cohomology up to semi-simplification (since the cohomology of local systems on $\M_2^\ct$ was only computed with this caveat in \cite{localsystemsA2}), but in this table there are actually no nontrivial extensions since none of the local systems have cohomology of different weights in the same degree. 

In Table \ref{tabell2} we give the tautological cohomology of $\M_{2,8}^\ct$, i.e.\ the first non-Gorenstein example. Here we don't specify the mixed Hodge structure (obviously it's all pure Tate). 

\begin{table}[t]\footnotesize
\begin{center}
\begin{tabular}{ | c || p{12.5cm} | }
\hline
$n$ & $H^\bullet(\M_{2,n}^\ct)$ \\
\hline
1 & ${L}^{2}s_{{1}}{t}^{4}+2\,Ls_{{1}}{t}^{2}+s_{{1}}$ \\
2 & ${L}^{3}s_{{2}}{t}^{6}+ ( {L}^{5}s_{{2}}+{L}^{4}s_{{1,1}} ) 
{t}^{5}+ ( 4\,{L}^{2}s_{{2}}+{L}^{2}s_{{1,1}} ) {t}^{4}+
 ( 4\,Ls_{{2}}+Ls_{{1,1}} ) {t}^{2}+s_{{2}}$ \\
3 & ${L}^{4}s_{{3}}{t}^{8}+ ({L}^{6}s_{{3}} + {L}^{6}s_{{2,1}}+{L}^{5}s_{{
2,1}}+{L}^{5}s_{{1,1,1}}) {t}^{7}+ ( 5\,{L}^{3}s_{{3}}+3\,{L}^{3}s_{{2,1}
} ) {t}^{6}+ ( {L}^{5}s_{{3}}+{L}^{5}s_{{2,1}}+2\,{L}^{4}s
_{{2,1}} +2\,{L}^{4}s_{{1,1,1}} ) {t}^{5}+
 ( 10\,{L}^{2}s_{{3}}+7\,{L}^{2}s_{{2,1}} ) {t}^{4}+
 ( 5\,Ls_{{3}}+3\,Ls_{{2,1}} ) {t}^{2}+s_{{3}}$  \\
4 & ${L}^{5}s_{{4}}{t}^{10}+ ( {L}^{7}s_{{4}}+{L}^{7}s_{{3,1}}+{L}^{7}s_{{2,2}}+{L}^{6}s_{{3,1}}+{L}^{6}s_{{2
,1,1}} ) {t}^{9}+ ( 
7\,{L}^{4}s_{{4}}+4\,{L}^{4}s_{{3,1}}+2\,{L}^{4}s_{{2,2}} ) {t}^
{8}+( 2\,{L}^{6}s_{{4}}+4\,{L}^{6}s
_{{3,1}}+2\,{L}^{6}s_{{2,2}}+{L}^{6}s_{{2,1,1}}+{L}^{6}s_{{1,1,1,1}}+{L}^{5}s_{{4}}+5\,{L}^{5}s_{{3,1}}+2\,{L}^{5}s_{{2,2}}+6\,{L}^{5}s_{{2,1,1
}}+2\,{L}^{5}s_{{1,1,1,1}}
 ) {t}^{7}+ ( 20\,{L}^{3}s_{{4}}+18\,{L}^{3}s_{{3,1}}+9\,{L}^{3}s_{{2,2}}+3\,{L}
^{3}s_{{2,1,1}} ) {t}^{6}+ ( {L}^{5}s_
{{4}}+{L}^{5}s_{{3,1}}+{L}^{5}s_{{2,2}} +3\,{L}^{4}s_{{3,1}}+{L}^{4}s_{{2,2}}+4\,{L}
^{4}s_{{2,1,1}}+{L}^{4}s_{{1,1,1,1}}) {t}^{5}+
 ( 20\,{L}^{2}s_{{4}}+18\,{L}^{2}s_{{3,1}}+9\,{L}^{2}s_{{2,2}}+3
\,{L}^{2}s_{{2,1,1}} ) {t}^{4}+ ( 7\,Ls_{{4}}+4\,Ls_{{3,1}}+2\,Ls_{{2,2}}
 ) {t}^{2}+s_{{4}}$ \\
 5   & ${L}^{6}s_{{5}}{t}^{12}+ ({L}^{8}s_{{5}}+{L}^{8}
s_{{4,1}}+ {L}^{8}s_{{3,2}}+{L}^{7}s_{{4,1}}+{L}^{7}s_{{3,1,1}} ) {t}^{11}+
 ( 8\,{L}^{5}s_{{5}}+6\,{L}^{5}s_{{4,1}}+3\,{L}^{5}s_{{3,2}}
 ) {t}^{10}+ 
 ( 4\,{L}^{7}s_{{5}}+7\,{L}^{7}s_{{4,1}}+6\,{L}^{7}s_{{3,2}}+3\,{L}^{7}s_{{3,1,1
}}+2\,{L}^{7}s_{{2,2,1}}+{L}^{7}s_{{2,1,1,1}}+{L}^{7
}s_{{1,1,1,1,1}}+2\,{L}^{6}s_{{5}}+9
\,{L}^{6}s_{{4,1}}+5\,{L}^{6}s_{{3,2}}+11\,{L}^{6}s_{{3,1,1}}+4\,{L}^{6}s_{{2,2,1}}+4\,{L}^{6}s_{{2,1,1,1}
} ) {t}^{9}+ 
( {L}^{7}s_{{3,2}}+{L}^{7}s_{{2,2,1}}+{L}^{7}s_{
{2,1,1,1}}+{L}^{7}s_{{1,1,1,1,
1}}+33\,{L}^{4}s_{{5}}+37\,{L}^{4}s_{{4,1}}+25\,{L}^{4}s_{{3,2}}+8\,{L}^{4}s_{{3,1,1}}+6\,{
L}^{4}s_{{2,2,1}} ) {t}^{8}+ 
( 4\,{L}^{6}
s_{{5}}+7\,{L}^{6}s_{{4,1}}+6\,{L}^{6}s_{{3,2}}+3\,{L}^{6}s_{{3,1,1}}+2\,{L}^{6}s_{{2,2,1}}+2\,{L}^{
6}s_{{2,1,1,1}}+2\,{L}^{6}s_{{1,1,1,1,1}}+3\,{L}^{5}s_{{5}}+15\,{L}^{5}s_{{4,1}}+10\,{L}^{5}s_{{3,2}}+21\,{L}^{5}s_{{3,1,1}}+10\,{L}^{5}s_{{2,2,1}}+
10\,{L}^{5}s_{{2,1,1,1}}+{L}^{5}
s_{{1,1,1,1,1}}
 ) {t}^{7}+ 
  ( 51\,{L}^{3}s_{{5}}+68\,{L}^{3}s_{{4,1}}+48\,{L}^{3}s_{{3,2}}+22\,{
L}^{3}s_{{3,1,1}}+{L}^{3}s_{{2,1,1,1}}+14\,{L}^{3
}s_{{2,2,1}} ) {t}^{6}+ ( {L}^{5}s_{{5}}+{L}
^{5}s_{{4,1}}+{L}^{5}s_{{3,2}}+4\,{L}^{4}s_{{4,1}}+2\,{L}^{4}s_{{3,2}}+6\,{L}^{4}s_{{3,1,1}}+2\,{L}^{4}s_{{2,2,
1}}+2\,{L}^{4}s_{{2,1,1,1}} ) {t}^{5}+
 ( 33\,{L}^{2}s_{{5}}+37\,{L}^{2}s_{{4,1}}+25\,{L}^{2}s_{{3,2}}+8\,{L}^{2}s_{{3,1,1}}+
6\,{L}^{2}s_{{2,2,1}} ) {t}^{4}+ ( 8\,Ls_{{5}}+6\,Ls_{{4,1}} +3\,
Ls_{{3,2}}) {t}^{2}+s_{{5}}$ \\
\hline 
\end{tabular}
\end{center}
\caption{Cohomology of $\M_{2,n}^\ct$ for $n \leq 5$. The equality $H^\even(\M_{2,n}^\ct) = \PH^\bullet(\M_{2,n}^\ct)$ fails for the first time when $n=5$, with the impure cohomology arising from classes in $H^2(\Sym^2\M_{1,1},\V)$ given by products of classes attached to Eisenstein series, and $\V$ is a Tate twist of a weight $4$ local system.}
\label{tabell1}
\end{table}

\begin{table}[t]\footnotesize
\begin{center}
\begin{tabular}{ | c || p{12.5cm} | }
\hline
$k$ & $\RH^k(\M_{2,8}^\ct)$ \\
\hline
0 & $s_8$  \\
2 & $13\,s_{{8}}+10\,s_{{7,1}}+8\,s_{{6,2}}+4\,s_{{5,
3}}+2\,s_{{4,4}}$ \\
4 & $87\,s_{{8}}+122\,s_{{7,1}}+127\,s_{{6,2}}+40\,s_{{6,1,1}}+90
\,s_{{5,3}}+52\,s_{{5,
2,1}}+37\,s_{{4,4}}+34\,s_{{4,3,1}}+16\,s_{{4,2,2}}+6\,s_{{3,3,2}} $ \\
6 & $307\,s_{{8}}+565\,s_{{7,1}}+695
\,s_{{6,2}}+319\,s_{{6,1,1}}+563\,s_{{5,3}}+
485\,s_{{5,2,1}}+53\,s_{{5,1,1,1
}}+234\,s_{{4,4}}+355\,s_{{4,3,1}}+183\,s_{{4,2,2}}+80\,s_{{4,2,1,1}}+91\,s_{{3,3,2}}+38\,s_{{3,3,1,1}}+32\,s_{{3,2,2,1
}}+4\,s_{{2,2,2,2}}$\\
8 & $578\,s_{{8}}+1194\,s_{{7,1}}+1578\,s_{{6,2}}+814\,s_{{6,1,1}}+1354\,s_{{5,3}}+1328\,s_{{5,2,1}}+202
\,s_{{5,1,1,1}}+569\,s_{{4,4}}+1018\,s_{{4,3,1}}+547\,s_{{4,2,2}}+326\,s_{{4,2,1,1}}+13\,s_{{4,1,1,1,1}}+290\,s_{{3,3,2}
}+156\,s_{{3,3,1,1}}+136\,s_{{3,2,2,1}}+16\,s_{{3,2,1,1,1}}+\mathbf{16\,s_{{2,
2,2,2}}}+4\,s_{{2,2,2,1,1}}$ \\ 
10 & $578\,s_{{8}}+1194\,s_{{7,1}}+1578\,s_{{6,2}}+814\,s_{{6,1,1}}+1354\,s_{{5,3}}+1328\,s_{{5,2,1}}+202
\,s_{{5,1,1,1}}+569\,s_{{4,4}}+1018\,s_{{4,3,1}}+547\,s_{{4,2,2}}+326\,s_{{4,2,1,1}}+13\,s_{{4,1,1,1,1}}+290\,s_{{3,3,2}
}+156\,s_{{3,3,1,1}}+136\,s_{{3,2,2,1}}+16\,s_{{3,2,1,1,1}}+
\mathbf{17\,s_{{2,2,2,2}}}+4\,s_{{2,2,2,1,1}} $ \\
12 &  $307\,s_{{8}}+565\,s_{{7,1}}+695
\,s_{{6,2}}+319\,s_{{6,1,1}}+563\,s_{{5,3}}+
485\,s_{{5,2,1}}+53\,s_{{5,1,1,1
}}+234\,s_{{4,4}}+355\,s_{{4,3,1}}+183\,s_{{4,2,2}}+80\,s_{{4,2,1,1}}+91\,s_{{3,3,2}}+38\,s_{{3,3,1,1}}+32\,s_{{3,2,2,1
}}+4\,s_{{2,2,2,2}}$\\
14 & $87\,s_{{8}}+122\,s_{{7,1}}+127\,s_{{6,2}}+40\,s_{{6,1,1}}+90
\,s_{{5,3}}+52\,s_{{5,
2,1}}+37\,s_{{4,4}}+34\,s_{{4,3,1}}+16\,s_{{4,2,2}}+6\,s_{{3,3,2}} $ \\
16 & $13\,s_{{8}}+10\,s_{{7,1}}+8\,s_{{6,2}}+4\,s_{{5,
3}}+2\,s_{{4,4}}$ \\
18 & $s_8$ \\
\hline 
\end{tabular}
\end{center}
\caption{The $\sym_8$-representations $\RH^\bullet(\M_{2,8}^\ct) = \PH^\bullet(\M_{2,8}^\ct)$, with failure of the Gorenstein property highlighted.}
\label{tabell2}
\end{table}

\newcommand{\RR}{\mathrm{R}}
\renewcommand{\R}{\mathbf{R}}
\renewcommand{\AA}{\overline{\A}}

\section{The nonvanishing of the Gysin map}\label{gysinsection}

In this section we will prove the nonvanishing of the Gysin map discussed in Remark \ref{remarkbefore}. Let us first give  a trivial reformulation of the result. Let $p \colon \A_1 \times \A_1 \to \A_2$ be the natural map, assigning to a pair of elliptic curves its product considered as a principally polarized abelian surface. Consider the local system $\V_{2a,2a}$ on $\A_2$. The pullback $p^\ast \V_{2a,2a}$ contains a single copy of the trivial local system $\Q$ as a summand. Choosing a map $\Q \hookrightarrow p^\ast\V_{2a,2a}$ (which is unique up to a nonzero scalar) allows one to define a Gysin map
$$ p_! \colon H^0(\A_1\times \A_1,\Q) \to H^2(\A_2,\V_{2a,2a})$$
for all $a$. (As the reader may have noticed, we are now omitting Tate twists -- the Hodge structures on both sides will no longer play any role.) The goal of this section will be to prove the nonvanishing of $p_!$ for all $a \neq 1$. 

\subsection{Reduction to the Baily--Borel boundary}
\newcommand{\pp}{\overline p}
Let $j \colon \A_2 \hookrightarrow \AA_2$ be the inclusion into the Baily--Borel compactification (also known as the Satake compactification). The space $\AA_2$ can be written set-theoretically as 
$$ \AA_2 = \A_2 \sqcup \A_1 \sqcup \A_0,$$
where $\A_0$ is a point. Let $i_0$ be the inclusion of $\A_0$ in $\AA_2$. Similarly, let $j' \colon \A_1 \times \A_1 \hookrightarrow \AA_1 \times \AA_1$ be the inclusion into the Baily--Borel compactification, and $i_0' $ the inclusion of the zero-dimensional boundary stratum (the product of the two cusps in $\AA_1 \times \AA_1$). 

We can consider the derived pushforwards into the respective Baily-Borel compactifications, $\RR j_\ast \V_{2a,2a}$ and $\RR j_\ast' \Q$, and take their stalks at the respective zero-dimensional strata. As in \cite[Subsection 3.2]{lemma3fold} one can define a Gysin map also between the stalks (we will discuss this further in Subsection \ref{hardertaketwo}), giving rise to a commutative diagram:
\begin{equation}
\begin{diagram}\label{maindiagram}
H^0(\A_1\times\A_1,\Q) & \rTo^{\cong} & i_0'^\ast j'_\ast \Q \\
\dTo_{p_!} & & \dTo_{p_{0!}} \\
H^2(\A_2,\V_{2a,2a}) & \rTo & i_0^\ast \RR^2  j_\ast \V_{2a,2a}
\end{diagram}\end{equation}
Clearly $j'_\ast\Q$ is the constant sheaf $\Q$ on $\AA_1 \times \AA_1$, and thus the upper horizontal arrow is an isomorphism. It follows that $p_!$ will be nonzero whenever $p_{0!}$ is, which is what we will actually prove.

\subsection{Harder's formula, take one} 

Let us first compute the stalks appearing in the right hand side of the diagram \eqref{maindiagram}. There is a useful and well known formula which calculates the restrictions to individual strata of  higher direct images of local systems in a Baily--Borel compactification. We will refer to it as \emph{Harder's formula}, since it seems that it does not have a standard name and that it was first obtained by Harder. The earliest proofs of Harder's formula that I am aware of prove it in the far more general settings of mixed Hodge modules \cite{looijengarapoport} and $\ell$-adic sheaves \cite{pinktheorem}, respectively. We state the result only in our particular situation; see the above references for the general case.

Let $G= \Sp(4)$, and let $K=U(2) \subset G(\R)$. Then $\A_2$ is a locally symmetric space for $G$:
$$ \A_2 \cong G(\Z) \backslash G(\R)/K.$$ 
By the general theory of the Baily--Borel compactification, the strata of $\AA_2$  correspond bijectively to $G(\Z)$-conjugacy classes of maximal parabolic subgroups of $G$. The zero-dimensional stratum  corresponds to  the Siegel parabolic subgroup $P \subset G$ of matrices of the form 
$$\begin{pmatrix}
 \ast & \ast &\ast &\ast \\
\ast & \ast &\ast &\ast \\
& & \ast & \ast \\
& & \ast & \ast
\end{pmatrix}.$$

 We write its Levi decomposition as $P=MU$, where $M \cong \GL(2)$ consists of block-diagonal matrices of the form $\left(\begin{smallmatrix}
A & \\ & A^{-t} \end{smallmatrix}\right)$ and $U$ is the additive group of $2\times 2$ symmetric matrices, sitting in the symplectic group as $\left(\begin{smallmatrix}
\mathbf 1 & B \\ & \mathbf 1 \end{smallmatrix}\right)$. 
Then Harder's formula asserts that
$$ i_0^\ast \RR^2 j_\ast \V_{2a,2a} \cong \bigoplus_{p+q=2} H^p(M(\Z),H^q(\mathfrak u,\V_{2a,2a})).$$
Here $\mathfrak u$ is the Lie algebra of $U(\R)$. 

 A famous theorem of Kostant expresses $H^q(\mathfrak u,\V_{2a,2a})$ as a representation of $M$. 
%
%
 For our purposes, it will be enough to know it as a representation of the derived subgroup $M^{\mathrm{der}} \cong \SL(2)$, whose representations we continue to denote $\V_k$, $k \geq 0$.  In lieu of  describing Kostant's theorem we quote from \cite[Table 2.3.4]{residual} that
$H^1(\mathfrak u,\V_{2a,2a}) \cong \V_{4a+2}.$
This is the only cohomology group we will need to consider.

Let also $H = \SL(2) \times \SL(2)$, which we consider as a subgroup of $ G$ via the inclusion
$$ \begin{pmatrix}
a & b \\ c & d 
\end{pmatrix},  \begin{pmatrix}
a' & b' \\ c' & d'  
\end{pmatrix} \mapsto \begin{pmatrix}
a & & b & \\
& a' & & b' \\
c & & d & \\
& c' && d' 
\end{pmatrix}.$$
 The corresponding locally symmetric space $H(\Z)\backslash H(\R)/K_H$ (where $K_H = H(\R) \cap K$) is isomorphic to $\A_1 \times \A_1$, and the map induced by $H \hookrightarrow G$ is precisely $p$. Let $P_H = P \cap H$, which is the product of the two Borel parabolic subgroups of $\SL(2)$. Let similarly $M_H = M\cap H$ (which is the diagonal torus in $H$), $U_H = U \cap H$, and let $\mathfrak u_H$ be the Lie algebra of $U_H(\R)$. We have already noted that $i_0'^\ast j'_\ast \Q \cong \Q$, but we can compute also with Harder's formula that 
$$ i_0'^\ast j'_\ast \Q \cong H^0({M_H}(\Z),H^0(\mathfrak u_H,\Q)) \cong \Q.$$

However, it is not clear at all from these descriptions what the map $p_{0!}$ is, or even why such a map should exist at all. To make sense of this, we need to understand the geometric content of Harder's formula. 

\subsection{Harder's formula, take two}\label{hardertaketwo}

By definition, $i_0^\ast \RR^k j_\ast \V$ is given by the direct limit of $H^k(V\cap \A_2,\V)$ where $V$ ranges over neighborhoods of the zero-dimensional stratum in $\AA_2$. The shape of Harder's formula is explained by the fact that a small enough deleted neighborhood $V \cap \A_2$ is homeomorphic to an $U(\R)/U(\Z)$-bundle over the space $M(\Z)\backslash M(\R)/K_M$. Now the cohomology of $U(\R)/U(\Z)$ coincides with the Lie algebra cohomology of $\mathfrak u$ by the van Est isomorphism (which is somewhat trivial in this case as the Lie algebra in question is abelian), and the cohomology of $M(\Z)\backslash M(\R)/K_M$ coincides with the group cohomology of $M(\Z)$ since $M(\R)/K_M$ is contractible. Finally, the Leray spectral sequence associated with this $U(\R)/U(\Z)$-bundle degenerates. This explains Harder's formula. See \cite[Section 6]{looijengarapoport} for more details. 

We can now understand the map $p_{0!}$. For every $V$ as in the previous paragraph, there is an oriented codimension $2$ embedding $V \cap (\A_1 \times \A_1) \hookrightarrow V \cap \A_2$. The direct limit of the corresponding Gysin maps defines $p_{0!}$. But there is also a Gysin map between the two Leray spectral sequences, defined by the two codimension one embeddings
$$ U_H(\R)/U_H(\Z) \hookrightarrow U(\R)/U(\Z) \qquad \text{and} \qquad M_H(\Z)\backslash M_H(\R)/K_{M_H} \hookrightarrow M(\Z)\backslash M(\R)/K_{M}$$
between the respective fibers and base spaces. This includes in particular a map
\begin{equation}
H^0({M_H}(\Z),H^0(\mathfrak u_H,\Q)) \to H^1(M(\Z),H^1(\mathfrak{u},\V_{2a,2a})) \label{hardergysin}
\end{equation} 
which we want to prove is nonzero. 

Consider now the following diagram with exact rows:
\begin{equation}
 \begin{diagram} 
1 & \rTo & \G_m & \rTo & M_H & \rTo^{\small\textrm{det}}  & \G_m & \rTo & 1 \\
&& \dTo & & \dTo & & \dEq && \\
1 & \rTo & M^{\mathrm{der}} & \rTo & M & \rTo^{\small\textrm{det}} & \G_m & \rTo & 1
\end{diagram}
\end{equation}
By assigning to each of the above groups $G$ its locally symmetric space $G(\Z)\backslash G(\R)/K$ the rows become fiber bundle sequences. We therefore get a map between fiber bundles over the base $$\G_m(\Z)\backslash \G_m(\R) /\pm 1 \cong \R_{>0} \times B\mathbf Z/2.$$ 
We also get a map between the Leray spectral sequences for the two fiber bundles. These spectral sequences compute the cohomology of $M_H(\Z)$ and $M(\Z)$ respectively, and become rather trivial in our case since $B\Z/2$ has nonzero cohomology only in degree $0$ (rationally), where it is given by the space of  $\Z/2$-invariants. We thus find a commutative diagram
\begin{equation}\label{Gysindiagram}
\begin{diagram}
H^0({M_H}(\Z),H^0(\mathfrak u_H,\Q)) & \rTo & H^1(M(\Z),H^1(\mathfrak{u},\V_{2a,2a})) \\
\dTo^\cong & & \dTo^\cong \\
H^0(\mathbf R_{>0} \times B\Z/2,\Q)^{\mathbf Z/2} & \rTo & H^1(\A_1,\V_{4a+2})^{\mathbf Z/2}
\end{diagram}
\end{equation}
where the superscripts denote $\mathbf Z/2$-invariants; the $\Z/2$-action is defined by the monodromy of the respective fiber bundles. Here $\A_1$ and $\R_{>0} \times B\Z/2$ arise as locally symmetric spaces for $M^{\mathrm{der}}$ and $\G_m$, respectively. Now $H^0(\mathfrak u_H,\Q)$ is the trivial representation of $\G_m \times \G_m$, which means that both factors $\Z/2$ act trivially on $H^0(\R_{>0},\Q)$. Thus we will know that the Gysin maps of diagram \eqref{Gysindiagram} are nonzero, whenever the Gysin map 
\begin{equation}\label{lastgysin}H^0(\mathbf R_{>0},\Q) \to H^1(\A_1,\V_{4a+2})\end{equation}
(which is $\Z/2$-equivariant!) is nonzero. The latter map is associated with the inclusion of $\R_{>0}$ into the upper half plane, $t \mapsto t \cdot i$. 

Note that we also need to specify a map from the constant sheaf $\Q$ on $\R_{>0}$ to the restriction of the local system $\V_{4a+2}$, which is of course a trivial local system of rank $4a+3$. This map is pinned down by the $\G_m$-action.  In general, the restriction of the representation $\V_{k}$ of $\SL(2)$ to its diagonal torus
$\G_m$ decomposes as the direct sum of $k+1$ eigenspaces, with eigenvalues
$$ -k, -k+2,\ldots, k-2, k. $$
Thus the restriction of $\V_{4a+2}$ to $\mathbf R_{>0}$ has a canonically defined $1$-dimensional summand corresponding to the eigenvalue $0$, and it is this summand which we identify with the constant sheaf $\Q$. 

The question of nonvanishing of the map $H^0(\mathbf R_{>0},\Q) \to H^1(\A_1,\V_{k})$ is related to classical theory of periods of cusp forms and modular symbols, as we now explain.


\subsection{The Eichler--Shimura isomorphism and periods of cusp forms} \label{eichlershimura}In this subsection we briefly recall results principally due to Eichler, Shimura and Manin. Everything here is well known and can be found e.g.\ in \cite[Part II]{langmodularforms} (although without the cohomological language). 

It will be slightly more convenient to consider the complexification of the rational local system $\V_k$ considered previously (although strictly speaking not necessary).  Let us make the identification
 $$ \V_k \otimes \C \cong \C[X,Y]_k,$$
 where the right hand side denotes the space of homogeneous polynomials in $X$ and $Y$ of degree $k$, with the action
$$ (P\vert \gamma)(X,Y) = P(aX+cY,bX+dY)$$ 
 for $P(X,Y) \in \C[X,Y]_k$ and $\gamma = \left(\begin{smallmatrix} a & b \\ c & d \end{smallmatrix}\right) \in \SL(2,\Z)$.

Let $f$ be a cusp form of weight $k+2$ for $\SL(2,\Z)$. Consider  the expression
$$ f(z) (zX + Y)^k \ud z,$$ 
which defines a holomorphic $1$-form in the upper half plane $\mathfrak H$ taking values in $\V_k \otimes \C$. It is easily checked that this $1$-form is closed and invariant under the action of $\SL(2,\Z)$, hence defines a cohomology class in
$H^1(\A_1,\V_k \otimes \C). $
In addition, these $1$-forms are actually rapidly decreasing at infinity, which implies that they define cohomology classes also in compactly supported cohomology. We write them as 
$$[f] \in H^1_c(\A_1,\V_k \otimes \C). $$ 
A remark is that a similar construction associates a cohomology class also to an antiholomorphic cusp form. 
The Eichler--Shimura theorem says that this construction defines an injection
$$ S_{k+2} \oplus \overline{S_{k+2}} \hookrightarrow H^1(\A_1,\V_k \otimes \C),$$
where $S_{k+2}$ (resp.\ $\overline{S_{k+2}}$) is the space of holomorphic (resp.\ antiholomorphic) cusp forms of weight $k+2$; moreover,  the cokernel of this injection is identified with the space of Eisenstein series of weight $k+2$.

As explained before, the inclusion $\G_m \hookrightarrow \SL(2)$ induces a map $\R_{>0} \to \mathfrak{H}$ of symmetric spaces, $t \mapsto t \cdot i$. We get a pullback map
$$ H^1_c(\A_1,\V_k \otimes \C) \to H^1_c(\R_{>0},\V_k\otimes \C) = \V_k \otimes \C $$
under which it is easy to check that 
$$ [f] \mapsto \sum_{n=0}^k \binom  k n  X^n Y^{k-n}  \int_{0}^{i\infty} f(z) z^n \ud z$$
for $f \in S_{k+2}$. Indeed, if we represent the compactly supported cohomology class associated to $f$ by the above rapidly decreasing differential form, then the pullback is given simply by integration. 

The integral $\int_{0}^{i\infty} f(z) z^n \ud z$ is called the $n$th period of $f$ and we denote it $r_n(f)$. If $f$ is a  Hecke eigenform, then its periods are related to the special values of the $L$-function attached to $f$: if the $L$-function is normalized to have a functional equation relating $L(f,s)$ and $L(f,k+2-s)$, then 
$$ L(f,n) = \frac{(-2\pi i)^n}{(n-1)!}r_{n-1}(f), \qquad n= 1,2, \ldots, k+1. $$

 As mentioned earlier, the restriction of $\V_k$ to $\G_m \subset \SL(2)$ splits as a sum of $k+1$ distinct $1$-dimensional representations, whose eigenvalues are
 $$ k, k-2, \ldots, -k+2, -k.$$   
We can  choose $X^k$, $X^{k-1}Y$, \ldots, $Y^k$ to be the eigenvectors in $\V_k \otimes \C$ corresponding to the above eigenvalues. Each such monomial $X^nY^{k-n}$ gives rise to a map $\C \hookrightarrow \V_k \otimes \C$ from the constant sheaf to the restriction of $\V_k \otimes \C$ to $\R_{>0}$, and thus to an evaluation
\begin{equation}\label{modsymbeq}
H^1_c(\A_1,\V_k \otimes \C) \to H^1_c(\R_{>0},\C) = \C
\end{equation}
which maps a cusp eigenform $f$ to $ \binom k n  r_n(f)$. These evaluations are classically called \emph{modular symbols}. The case we consider is the eigenvalue $0$, which corresponds to the middle monomial $X^{k/2}Y^{k/2}$ and to the central critical value of the $L$-function.  
 
	 \subsection{Putting it all together}

We are now in a position to prove the main result of this section. 
\begin{thm} \label{gysintheorem} The Gysin map $p_! \colon H^0(\A_1 \times \A_1,\Q) \to H^2(\A_2,\V_{2a,2a})$ is nonzero for all $a \neq 1$. \end{thm}

\begin{proof}The case $a=0$ is easy, so assume $a \geq 2$.  We have already noted that the map $p_!$ is nonzero whenever $p_{0!}$ is nonzero, as is clear from the diagram \eqref{maindiagram}. By expressing the source and target of $p_{0!}$ using Harder's formula, we saw that it would be enough to show that the Gysin map in Equation \eqref{hardergysin} is nonzero. By considering the diagram \eqref{Gysindiagram} we then saw that this would follow from the nonvanishing of the Gysin map \eqref{lastgysin}. 

Tensoring our cohomology groups by $\C$ and applying Poincar\'e duality, we see that the last statement is equivalent to the nonvanishing of the modular symbol in Equation \eqref{modsymbeq}, where $k=4a+2$. As explained in Subsection \ref{eichlershimura} this modular symbol will map a Hecke eigenform $f$ of weight $4a+4$ to its central $L$-value $L(f,2a+2)$ (up to a nonzero scalar). Thus we have finally reduced to the following assertion: for any weight $w \equiv 0 \pmod 4$, $w \geq 12$, there is a Hecke eigenform of weight $w$ whose central $L$-value is nonzero. This last fact is a special case of \cite[Corollary of Theorem 2]{kohnenhalfintegral} (or alternatively of \cite[Corollary 2]{kohnenzagier}). This finishes the proof. \end{proof}

\begin{rem} \label{remarkafterproof} Recall from \cite{localsystemsA2} that $H^4_c(\A_2,\V_{2a,2a})$ has a basis corresponding to the normalized cusp eigenforms for $\SL(2,\Z)$ of weight $4a+4$. Moreover, this cohomology group is the direct sum of the \emph{cuspidal cohomology} and the \emph{residual Eisenstein cohomology}, where the former is spanned by those cusp forms whose central critical value vanishes, and the latter is spanned by those with nonzero central value. It is natural to think of Theorem \ref{gysintheorem} as the conjunction of two separate assertions: firstly that the map $H^4_c(\A_2,\V_{2a,2a}) \to H^4_c(\A_1 \times \A_1, \Q)$ does not vanish on the residual Eisenstein subspace, and  secondly that the residual Eisenstein subspace is nonempty. The latter fact is where we need the input from Kohnen--Zagier in the proof of Theorem \ref{gysintheorem}. We remark also that $H^4_c(\A_2,\V_{2a,2a}) \to H^4_c(\A_1 \times \A_1, \Q)$ will in fact vanish on the subspace of cuspidal cohomology, cf.\ \cite{ash-ginzburg-rallis}. \end{rem}

\begin{rem} A statement nearly equivalent to Theorem \ref{gysintheorem} had previously been obtained in \cite[Lemma 14]{qiu}. In fact, Qiu proves the nonvanishing of an analogously defined Gysin map on the residual Eisenstein subspace assated with the Siegel parabolic subgroup $P$ (cf.\ the preceding remark) for an \emph{arbitrary} Siegel threefold (as opposed to $\A_2$, which is the special case of the full modular group).  However, he restricts his attention to the case of cohomology with constant coefficients, so we could not directly quote his result. Qiu's proof uses Harder's formalism of Eisenstein cohomology to represent residual Eisenstein classes as differential forms, and what he proves is that the integral of such a differential form over $\A_1 \times \A_1$ (or an analogously defined cycle) is nonzero. Compared to the proof given here, the computations involving Harder's formula and Leray spectral sequences are replaced by explicit manipulations of adelic integrals and Fubini's theorem; the classical theory of period integrals of cusp forms is replaced by the Jacquet--Langlands theory for $\GL(2)$. Qiu has informed me be that his proof would generalize to the case of local coefficients as well.  \end{rem}

\begin{rem}\label{centralvalue} Famous results of Manin and Shimura imply that  the special values $L(f,n)$ admit a factorization into a `transcendental part' and an `algebraic part'. The transcendental part is expressed in terms of certain numbers $\omega_\pm(f)$ and powers of $\pi$ (and is always nonzero). The algebraic part is an algebraic number, as the name suggests, and the action of $\mathrm{Gal}(\overline \Q/\Q)$ on the set of Hecke eigenforms is compatible with the natural action on the algebraic parts of the special values.  Now a conjecture of Maeda asserts that all Hecke eigenforms for $\SL(2,\Z)$ of given weight form a single Galois orbit. In particular, it would follow that if a single cusp eigenform of some weight $k$ has a vanishing/nonvanishing central value, then the same should be true for \emph{all} eigenforms of this weight.  Therefore the decomposition into cuspidal and residual Eisenstein cohomology discussed in Remark \ref{remarkafterproof} would become a bit silly, assuming Maeda's conjecture: the cuspidal subspace should always vanish. Maeda's conjecture is known to be true for all weights up to 14000 \cite{ghitza-mcandrew}. \end{rem}

\printbibliography

\end{document}